\documentclass[12pt]{article}
\usepackage{amsfonts}
\usepackage{amsmath}

\usepackage{anysize} 
\marginsize{3cm}{3cm}{2cm}{2.5cm} 

\allowdisplaybreaks
\newtheorem{theorem}{Theorem}

\newtheorem{proposition}[theorem]{Proposition}

\newenvironment{proof}[1][Proof]{\noindent\textbf{#1.} }{\ \rule{0.5em}{0.5em}}

\begin{document}

\title{\textbf{Global existence and persistence of mass for a nonlinear equation with fractional Laplacian}}
\author{\textbf{Eric Ruvalcaba-Robles} \\
Universidad Aut\'{o}noma de Aguascalientes\\
Departamento de Matem\'{a}ticas y F\'{\i}sica\\
Aguascalientes, Aguascalientes, Mexico.\\
\texttt{eric\_ruvalcaba\_robles@yahoo.com.mx}\bigskip\\
\textbf{Jos\'{e} Villa-Morales}\\
Universidad Aut\'{o}noma de Aguascalientes\\
Departamento de Matem\'{a}ticas y F\'{\i}sica\\
Aguascalientes, Aguascalientes, Mexico.\\
\texttt{jvilla@correo.uaa.mx}}
\date{}
\maketitle

\begin{abstract}
In this paper we study the partial differential equation
\begin{equation}\label{edpa}
	\begin{split}
	   	\partial_tu &= k(t)\Delta_\alpha u - h(t)\varphi(u),\\
	   	u(0) &= u_0.
	\end{split}
\end{equation}
Here $\Delta_\alpha$ is the fractional Laplacian, $k,h:[0,\infty)\to[0,\infty)$ are continuous functions and $\varphi:\mathbb{R}\to[0,\infty)$ is a convex differentiable function. If $u_0\in C_b(\mathbb{R}^d)\cap L^1(\mathbb{R}^d)$ we prove that \eqref{edpa} has a classical global solution and is non-negative if $u_0\geq0$. Imposing some restrictions on the parameters we prove that  the mass $M(t)=\int_{\mathbb{R}^d}u(t,x)dx$, $t>0$, of the system $u$ does not vanish in finite time, moreover we see that $\lim_{t\to\infty}M(t)>0$, under the restriction $\int_0^\infty h(s)ds<\infty$. A comparison result is also obtained for non-negative solutions, and as an application we get a better condition when $\varphi$ is a power function.

\vspace{0.5cm} \textbf{Mathematics Subject Classification (2010).} Primary
35K55, 35K45; Secondary 35B40, 35K20.

\end{abstract}

\section{Introduction and statement of results}

In this paper we study the asymptotic behavior of the solutions to the nonlinear partial differential equation
\begin{equation}\label{edps}
	\begin{split}
		\partial_tu(t,x) &= k(t) \Delta_{\alpha} u(t,x) - h(t)\varphi\left(u(t,x)\right),\enspace 						(t,x)\in(0,\infty)\times\mathbb{R}^d,\\
		u(0,x) &= u_0(x), \enspace x\in\mathbb{R}^d,
	\end{split}
\end{equation}
where $\Delta_\alpha = -(-\Delta)^{\alpha/2}$, $0<\alpha<2$, is the $\alpha$-Laplacian (or fractional Laplacian).

The study of partial differential equations (PDE) with fractional diffusion begins with the work of Sugitani \cite{Su}. He studied the explosive behavior of the solutions to the equation
\begin{equation}\label{spsps}
	\begin{split}
		\partial_tu &= \Delta_\alpha u + \varphi(u),\\
		u(0,x) &= u_0(x), \;x\in\mathbb{R}^d.
	\end{split}
\end{equation}

Since Sugitani's work there have been many generalizations of \eqref{spsps}. In fact, in some case the reaction term is multiplicative perturbed by a time dependent function or furthermore if the parameter $\alpha$ is changed then some systems of fractional PDE are considered, see for example \cite{AV2}, \cite{GK}, \cite{PV} and the references there in.

On the other hand, the importance of modeling phenomena using fractional Laplacian is increasing. For example, they arise in fields like mathematical finance, molecular biology, hydrodynamics, statistical physics \cite{SZF}, also they arise in anomalous growth of certain fractal interfaces \cite{MW}, overdriven detonations in gases \cite{Cla} or anomalous diffusion in semiconductor growth \cite{Wo}.

Therefore, our first concern is to demonstrate the existence of classical solutions of (\ref{edps}).

\begin{theorem}\label{main}
Let us assume the following hypotheses:
\begin{itemize}
	\item [$(a)$]  $\varphi:\mathbb{R}\to[0,\infty)$ is a convex differentiable function, $\varphi(0)=0$ and 					$\varphi(x)>0$ if $x\neq0$.
	\item [$(b)$] $h:[0,\infty)\to[0,\infty)$ is a continuous function with $h(x)>0$ if $x\neq0$.
	\item [$(c)$] $k:[0,\infty)\to[0,\infty)$ is a continuous function.
	\item  [$(d)$] $\mathop{\sup}\limits_{t\in[0,T]}\int_0^t\left(\int_s^tk(r)dr\right)^{-1/\alpha}ds							<\infty,\;\forall$ $ T>0$.
	\item  [$(e)$] $u_0\in C_b(\mathbb{R}^d)\cap L^1(\mathbb{R}^d)$.
\end{itemize}
Under these conditions there exists a unique solution $u\in C_b^{1,2}\left((0,\infty)\times\mathbb{R}^d\right)$ of \eqref{edps}. Moreover, if $u_0\geq0$ then $u\geq0$.
\end{theorem}

Since the fractional Laplacian $\Delta_\alpha$ is an integral operator (see Section \ref{secpreres}) most of the solutions of \eqref{edps} are interpreted in a weak sense. Thanks to an integral representation of $\Delta_\alpha$, given in \cite{DI}, it is possible to extend the domain of $\Delta_\alpha$. Using such extension is proved in \cite{DI} that certain Hamilton-Jacobi equation, perturbed by $\Delta_\alpha$, has solutions in $C_b^2\left((0,\infty)\times\mathbb{R}^d\right)$, when $\alpha\in(1,2)$. In our context the regularity of the solutions is mainly attributed to the above conditions $(a)$ and $(d)$. In particular, when $h\equiv 1$ and $k\equiv 1$ we obtain the same condition as in \cite{DI}.

Using the corresponding Duhamel equation associated with \eqref{edps} we prove the existence of local solutions to \eqref{edps}, through the Banach fix point principle. Then we show that such local solution can be extended (this is a usual way to get a global solution, see \cite{AB}). 

Otherwise, we are interested in the asymptotic behavior of solutions of \eqref{edps} when the initial datum is non-negative. Here the main ingredient in the proof of the positivity of the solutions is a reverse maximum principle obtained in \cite{DI}. The difficulty in the application of such principle is that the extremes have to be global, no local as in the classical case. To deal with this problem is proved first that for fix times the solutions vanish at space infinity (see \eqref{dppde}).

\begin{theorem}\label{pcasi}
Let us assume the hypotheses of Theorem \ref{main}. Let $u$ be the solution of \eqref{edps} with 	$u_0\geq0$ and $||u_0||_1>0$. For each $t\geq0$ we set
	$$M(t) = \int_{\mathbb{R}^d}u(t,x)dx.$$
Then
\begin{equation}\label{epmsp}
	M(t) \geq ||u_0||_1 \mathop{\exp}\left\{-\frac{\varphi(||u_0||_u)}{||u_0||_u}\int_0^t h(s)ds					\right\},\;\forall t\geq0.
\end{equation}
Moreover, the limit $\mathop{\lim}_{t\to\infty}M(t)=M(\infty)$ exists and $M(\infty)>0$ if $\int_0^\infty h(s)ds<\infty$.
\end{theorem}

We interpret $M(t)$ as the total mass of the system $u$ at time $t\geq0$. If $\int_0^\infty h(s)ds<\infty$, then $h$ vanishes at infinity ($\lim_{t \rightarrow \infty} h(t)=0$). This means that the contribution of the reaction term in \eqref{edps} is very small for large time, in such a way that the system $u$ does not vanish, almost everywhere, at finite time. In fact, this is because the solution is positive and the contribution of the negative term is small. Moreover the system persists at infinity. As we will see the proof of Theorem \ref{pcasi} uses strongly the convexity of $\varphi$.

By $K(t)$, $t\geq 0$, we mean the integral $\int_{0}^{t}k(s)ds$. Imposing an additional condition, on the multiplicative noise of the fractional diffusion term of \eqref{edps}, we have the rate of convergence in $L^p$ norm of the system $u(t)$, when $t\to\infty$:

\begin{theorem}\label{tcpa}
Assume the hypotheses of Theorem \ref{pcasi}. If $\int_1^\infty k(s)ds<\infty$, then
	$$\mathop{\lim}\limits_{t\to\infty}K(t)^{\frac{d}{\alpha}\left(1-\frac{1}{p}\right)}||u(t)-M(\infty) 			p(K(t))||_p = 0,$$
for each $p\geq1$.
\end{theorem}

As we already mentioned most of the works in the literature deal with weak solutions of \eqref{edps}. For example, in \cite{FK} is studied the decay of mass of the equation \eqref{edps} in the case $k\equiv1, h\equiv1$ and $\varphi(x)=x^\beta$, $\beta>1$. More generally, in \cite{JS} is studied the decay of mass when $k(t)=t^\sigma$, $\sigma\geq 0$, $h:[0,\infty)\to[0,\infty)$ is continuous and $\varphi(x)=x^\beta, \beta>1$.
In both works the existence, in the mild sense, of the respective solutions is assumed. A basic ingredient in their proof of mass decay is the positivity of the solutions, as we could not find a specific reference for such result we include a proof in Theorem \ref{main}. 
The positivity property of the solutions of \eqref{edps}, when the initial datum is non-negative, is close related with the maximum principle, and such principle is commonly applied for classical solutions (see for example \cite{AB} or \cite{Pi}). This is one of the reasons we need classical solutions of \eqref{edps} instead of mild solutions to ensure the positivity of such solutions.

If $u$ and $v$ are the solutions of \eqref{edps} with initial conditions $u_0$ and $v_0$, respectively, then $v_0\geq u_0$ implies $v\geq u$ (see Theorem \ref{cpth} below). As a consequence of this comparison result, and assuming that $\varphi$ is a power function, we can weak the condition $\int_0^\infty h(s)ds<\infty$ in Theorem \ref{pcasi}. The precise condition is given in:

\begin{theorem}\label{cppo}
Let us assume the hypotheses of Theorem \ref{main}. Let $u$ be the solution of \eqref{edps} with $u_0\geq0$ and $||u||_1>0$. If $\varphi(x)=|x|^\beta$, $\beta>1$, and
	$$\int_1^\infty h(s)K(s)^{-\frac{d}{\alpha}\left(\beta-1\right)}ds<\infty,$$
then there exists $M(\infty)\in(0,\infty)$ such that $\lim_{t\to\infty}M(t)=M(\infty)$.
\end{theorem}

The previous result is a generalization of Theorem 3.1 in \cite{JS}, since the continuity of $h:[0,\infty)\to[0,\infty)$ implies $\int_0^1 h(s)ds<\infty$, an hypothesis required in \cite{JS}.

Notice that $K$ increasing implies
	$$\int_1^\infty h(s)K(s)^{-\frac{d}{\alpha}\left(\beta-1\right)}ds \leq K(1)^{-\frac{d}{\alpha}						\left(\beta-1\right)}\int_1^\infty h(s)ds.$$
On the other hand, if we take $h \equiv 1$ and 
	$$k(t)=e^t,\;\forall t>0,$$
then $\int_0^\infty h(t)dt =  \infty$, and
	$$\int_1^\infty h(s)K(s)^{-\frac{d}{\alpha}\left(\beta-1\right)}ds = \int_1^\infty(e^s-1)^{-\frac{d}				{\alpha}\left(\beta-1\right)}ds < \infty.$$
This means that in the particular case $\varphi(x)=|x|^\beta$, $\beta>1$, we have a better condition for the convergence of $\mathop{\lim}_{t\to\infty}M(t)=M(\infty)$.

The paper is organized as follows. In Section \ref{secpreres} we introduce the fractional Laplacian, $\Delta_\alpha$, and the fundamental solution $p(t,x)$ of $\partial_t u=\Delta_\alpha u$, some of their properties are also stated. In Section \ref{secprofr} we give the proof of the theorems.

\section{Preliminary results}\label{secpreres}

If $X=\{X(t):t\geq 0$\} is a symmetric $\alpha$-stable process then $\{p(t,x):t>0\}$ is its transition density. An analytical way to define the functions $p(t,x)$ is through their Fourier transform
\begin{equation}\label{dtfd}
	\int_{\mathbb{R}^d} e^{z\cdot x i}p(t,x)dx =e^{-t||z||^\alpha},\enspace\forall t>0, z\in\mathbb{R}^d,
\end{equation}
where $\cdot$ and $||\cdot||$ are the inner product and Euclidean norm in $\mathbb{R}^{d}$, respectively.

Moreover $p(t,x)$ has the following properties.

\begin{proposition}\label{prop1}
Let $p(t,x)$ be defined by \eqref{dtfd}.
\begin{itemize}
	\item [$(a)$] For each $t>0$,
		\begin{equation}\label{ppu}
			\int_{\mathbb{R}^d}p(t,x)dx = 1.
		\end{equation}
	\item [$(b)$]  For each $t$, $\tilde{t}>0$,
		\begin{equation}\label{spp}
			p(t+\tilde{t}) = p(t)\ast p(\tilde{t}).
		\end{equation}
	\item [$(c)$]  For each $\mu\geq1$ there exists a constant $c=c(\alpha,\mu)>0$ such that
	\begin{equation}
	||p(t)||_\mu \leq c\;t^{-\frac{d}{\alpha}\left(1-\frac{1}{\mu}\right)},\;\forall t>0.\label{dnpv}
	\end{equation}
	\item  [$(d)$] The function $t\mapsto p(t,\cdot)\in L_1(\mathbb{R}^d)$ is continuous.\label{pcec}
	\item [$(e)$]  The function $(t,x)\mapsto p(t,x)$ is in $C^\infty\left((0,\infty)\times\mathbb{R}^d\right)$.					\label{pdiip}
	
	\item [$(f)$]  For $t>0$ there exists a constant $c=c(\alpha)>0$ such that
		\begin{equation}\label{pdd}
			||\partial_xp(t)||_1 \leq c\;t^{-1/\alpha},
		\end{equation}	
		and
		\begin{equation}\label{pesd}
			||\partial_x^2p(t)||_1 \leq c\;t^{-2/\alpha}.
		\end{equation}
	\item [$(g)$]  For every $f\in L^1(\mathbb{R}^d)$, we have 
			$$\mathop{\lim}\limits_{t\to\infty}||p(t)\ast f - Mp(t)||_1 = 0,$$
		where $M=\int_{\mathbb{R}^d}f(x)dx$.
\end{itemize}
\end{proposition}
\begin{proof}
For the proof of $(a)$, $(b)$, $(e)$, and $(f)$ see \cite{D6V}. The statement $(d)$ is in \cite{DI}, $(c)$ and $(g)$ are proved in \cite{AV} and \cite{JS}, respectively. \hfill
\end{proof}

\vspace{0.5cm}
The stochastic process $X$ induces a strongly continuous semigroup $\{T_t:t\geq0\}$ whose infinitesimal generator is the closed operator $\Delta_\alpha$ (see \cite{EK}).

In what follows we denote by $B_r$ the ball of $\mathbb{R}^d$ centered at the origin and of radius $r>0$ and by $\mathcal{S}(\mathbb{R}^{d})$ the Schwartz space on $\mathbb{R}^d$.

An integral representation for $\Delta_\alpha$, $0<\alpha<2$, is given in \cite{DI}: For all $f\in \mathcal{S}(\mathbb{R}^d)$, all $x\in\mathbb{R}^d$ and all $r>0$,
\begin{equation}\label{dexpla}
	\begin{split}
		\Delta_\alpha f(x) = -c&\left(\int_{B_r}\frac{f(x+z)-f(x)-\nabla f(x)\cdot z}{||z||^{d+\alpha}}dz					\right.\\
		&\left.+\int_{\mathbb{R}^d\setminus B_r}\frac{f(x+z)-f(x)}{||z||^{d+\alpha}}dz\right),
	\end{split}
\end{equation}
where $c=c(\alpha,d)>0$ is a constant.

The expression \eqref{dexpla} for $\Delta_\alpha$ allows to define $\Delta_\alpha f\in C_b(\mathbb{R}^d)$ for each $f\in C_b^2(\mathbb{R}^d)$, moreover this extension is continuous, see \cite{DI}.

\begin{proposition}\label{prop2}
Let $\Delta_\alpha$, $0<\alpha<2$, be the $\alpha$-Laplacian with domain $C_b^2(\mathbb{R}^d)$.
\begin{itemize}
	\item [$(a)$] If $f\in C_b^2(\mathbb{R}^d)$ and $\tilde{x}$ is a global minimum of $f$, then
		\begin{equation}\label{pqpmg}
			\Delta_\alpha f(\tilde{x}) \leq 0.
		\end{equation}
	\item  [$(b)$] If $f\in C_b^2(\mathbb{R}^d)$, then
		\begin{equation}\label{ddct}
			\partial_t\left(p(t)\ast f(x)\right) = \Delta_\alpha\left(p(t)\ast f\right)(x),\;\forall t>0.
		\end{equation}	
\end{itemize}
\end{proposition}
\begin{proof}
The statement $(a)$ is a trivial consequence of Theorem 2 in \cite{DI}. The formula \eqref{ddct} is proved in \cite{D6V} when $d=1$. When $f\in S(\mathbb{R}^d)$ the expression \eqref{dexpla} implies \eqref{ddct}, the general case is proved using a density argument (as in the proof of Proposition 1 in \cite{DI}).\hfill
\end{proof}

\section{Proof of the results}\label{secprofr}

By $B([0,T]\times\mathbb{R}^d)$ we denote the space of all real-valued bounded measurable functions defined on $[0,T]\times\mathbb{R}^d$. It is a Banch space with the norm	
$$|||u|||=\text{sup}\left\{|u(t,x)|:(t,x)\in[0,T]\times\mathbb{R}^d\right\}.$$
The subspace of continuous bounded functions will be denoted by $C_b([0,T]\times\mathbb{R}^d)$. In particular, if $u\in C_{b}(\mathbb{R}^{d})$ we denote, the uniform norm, by $||u||_{u}$.

In what follows we denote by $c$ a positive constant whose specific value is unimportant and it may change from place to place. Given the continuous functions $h,k:[0,\infty) \rightarrow [0,\infty)$, for each $t\geq 0$ we set 
$$H(t) = \int_{0}^t h(r)dr,\enspace K(t) = \int_{0}^t k(r)dr.$$
	
The Duhamel equation associated to \eqref{edps} is
\begin{equation}\label{eqint}
	u(t,x) = p(K(t))\ast u_0(x) - \int_0^t h(s)p(K(s,t))\ast\varphi(u(s))(x)ds,
\end{equation}
where
	$$K(s,t) = \int_s^t k(r)dr, \enspace 0<s<t.$$
	
\begin{proof}[Proof of Theorem \ref{main}]
The proof will be given in several steps.

\vspace{0.3cm}
{\it Step one} (existence of a local mild solution): Let us define the function $F:B([0,T]\times\mathbb{R}^d)\to B([0,T]\times\mathbb{R}^d)$, as
	$$(Fu)(t,x) = p(K(t))\ast u_0(x) - \int_0^th(s)p(K(s,t))\ast\varphi(u(s))(x)ds.$$
If $u\in\left\{v\in B\left([0,T]\times\mathbb{R}^d\right):|||v|||\leq R\right\}$, $R>0$, then \eqref{ppu}	implies
\begin{align*}
	|(Fu)(t,x)| &\leq \int_{\mathbb{R}^d} p(K(t),y-x)||u_0||_{u}dy + \int_0^t h(s) \int_{\mathbb{R}^d}					p(K(s,t),y-x)\varphi(|||u|||)dyds\\
	&\leq ||u_0||_{u} + \varphi(R)\int_0^t h(s)ds.
\end{align*}
Therefore
\begin{equation}\label{fenbr}
	|||Fu||| \leq ||u_0||_{u} + \varphi(R)H(T).	
\end{equation}
	
On the other hand, if $u,\tilde{u}\in\left\{v\in B\left([0,T]\times\mathbb{R}^d\right):|||v|||\leq R\right\}$, then
	$$|(Fu)(t,x)-(F\tilde{u})(t,x)| \leq \int_0^t h(s)p(K(s,t))\ast|\varphi(u(s))-\varphi(\tilde{u}					(s))|(x)ds.$$
			
The convexity of $\varphi$ implies (see Theorem 14.5 in \cite{Y})
	$$\frac{\varphi(b)-\varphi(a)}{b-a} \leq (D_l\varphi)(b), \text{ if }a<b,$$
where $D_l$ is the left hand side derivative of $\varphi$. Using this and $u(s,x),\tilde{u}(s,x)\leq R$, we get
	$$|\varphi(u(s,x))-\varphi(\tilde{u}(s,x))| \leq (D_l\varphi)(R)|u(s,x)-\tilde{u}(s,x)|.$$
Hence
\begin{equation}\label{fesco}
	|||Fu-F\tilde{u}||| \leq (D_l\varphi)(R)H(T)|||u-\tilde{u}|||.
\end{equation}

If we take $R^\star = ||u_0||_{u}+1$ and
\begin{equation}
T^\star < H^{-1}\left(\frac{1}{1+\varphi(R^\star)+(D_l\varphi)(R^\star)}\right),
\end{equation}
we see that \eqref{fenbr} and \eqref{fesco} implies the hypotheses of the Banach contraction principle. Therefore there exists a unique solution $u\in B([0,T]\times\mathbb{R}^d)$ of (\ref{eqint}).
	
\vspace{0.3cm}		
{\it Step two} (continuity of $u$): Now we will see $u_0\in C_b(\mathbb{R}^d)$ implies $u\in C_b((0,T^\star]\times\mathbb{R}^d)$. The first thing we are going to do is proving that $u$ is continuous uniformly in $x$. To this end we use the integral representation \eqref{eqint} of $u$. The continuity in $t$ uniformly in $x$ of the first term in the right hand side of equality \eqref{eqint} is an immediate consequence of \eqref{pcec}. Now let us study the function
\begin{equation}\label{sfia}
	(t,x) \mapsto \int_0^t h(s)p(K(s,t))\ast\varphi(u(s))(x)ds.
\end{equation}
Choose an arbitrary and fix point $\tilde{t}\in[0,T^\star]$. Let $\eta\in\mathbb{R}$ be such that $0\leq\tilde{t}+\eta\leq T^\star$,
\begin{align}
	& \left|\int_0^{\tilde{t}+\eta} h(s)p(K(s,\tilde{t}+\eta)\ast\varphi(u(s))(x)ds - \int_0^{\tilde{t}}			 h(s)p(K(s,\tilde{t}))\ast\varphi(u(s))(x)ds\right|\nonumber\\
	\leq & \left|\int_0^{\tilde{t}+\eta} h(s)p(K(s,\tilde{t}+\eta))\ast\varphi(u(s))(x)ds -						\int_0^{\tilde{t}+\eta} h(s)p(K(s,\tilde{t}))\ast\varphi(u(s))(x)ds\right|\nonumber\\
	&+\left|\int_0^{\tilde{t}+\eta} h(s)p(K(s,\tilde{t}))\ast\varphi(u(s))(x)ds -\int_0^{\tilde{t}} 				h(s)p(K(s,\tilde{t}))\ast\varphi(u(s))(x)ds\right|\nonumber\\
	\leq & \varphi(|||u|||)\left\{\int_{\tilde{t}-|\eta|}^{\tilde{t}+|\eta|} h(s)ds + \int_0^T h(s)||			p(K(s,\tilde{t}+\eta))-p(K(s,\tilde{t}))||_1ds\right\}.\label{cuets}
\end{align}
Using $(d)$ of Proposition \ref{prop1} we deduce  
	$$t \mapsto p(t,\cdot)\in L^1(\mathbb{R}^d)$$
is uniformly continuous on $[0,K(T^\star)]$. This implies the right hand side of \eqref{cuets} goes to $0$, when $\eta\to 0$. Hence we have the continuity of $u$ uniformly in $x$.
		
Fix an arbitrary point $(\tilde{t},\tilde{x})\in(0,T^\star]\times\mathbb{R}^d$. The properties of the convolution operator implies
	$$x \mapsto p(K(\tilde{t}))\ast u_0(x),$$
is a continuous function, and from the dominated convergence theorem we deduce that the function
	$$x \mapsto \int_0^{\tilde{t}} h(s)p(K(s,\tilde{t}))\ast\varphi(u(s))(x)ds$$
is continuous. From \eqref{eqint} we get
\begin{eqnarray*}
	&&\hspace{-0.8cm}|u(t,x)-u(\tilde{t},\tilde{x})|\\
	&\leq& p(K(t))\ast u_0(x)-p(K(\tilde{t}))\ast u_0(x)|\\
&&+\left| \int_0^t h(s)p(K(s,t))\ast\varphi(u(s))(x)ds - \int_0^{\tilde{t}} h(s)p(K(s,\tilde{t}))\ast			\varphi(u(s))(x)ds\right| \\
	&&+|p(K(\tilde{t}))\ast u_0(x)-p(K(\tilde{t}))\ast u_0(\tilde{x})|\\
	&&+\left| \int_0^{\tilde{t}} h(s)p(K(s,\tilde{t}))\ast\varphi(u(s))(x)ds - \int_0^{\tilde{t}} 					h(s)p(K(s,\tilde{t}))\ast\varphi(u(s))(\tilde{x})ds\right|.
\end{eqnarray*}

The previous observation and the continuity of $u$ in $t$ uniformly in $x$ implies that the right hand side of the above inequality goes to $0$, when $(t,x)\to(\tilde{t},\tilde{x})$.

\vspace{0.3cm}
{\it Step three} (spatial regularity of $u$): If $u_0\in C_b^2(\mathbb{R}^d)$ we will see that $u(t)\in C_b^2(\mathbb{R}^d)$, $t>0$. If $u$ were differentiable we would have, by \eqref{eqint},
	$$\partial_x u(t,x) = \partial_xp(K(t))\ast u_0(x)-\partial_x \int_0^t h(s)p(K(s,t))\ast\varphi(u(s))			(x)ds.$$
We are going to see that each term, in the above expression, is in fact differentiable. The differentiability of $p(K(t))\ast u_0(\cdot)$ follows from $(e)$ of Proposition \ref{prop1} and Lema 2 in \cite{DI}. To prove the differentiability of the second term we consider the function			
	$$x \mapsto \int_0^{t-\delta} h(s)p(K(s,t))\ast\varphi(u(s))(x)ds,$$
where $\delta>0$. The convolution operator properties and \eqref{pdd} yields
	$$\left|\partial_x\left(p(K(s,t))\ast\varphi(u(s))\right)(x)\right| \leq \varphi(|||u|||)
		K(t-\delta,t)^{1/\alpha}.$$
Therefore, see Theorem 2.27 in \cite{Folland},
\begin{equation}\label{igai}
	\partial_x\int_0^{t-\delta}h(s)p(K(s,t))\ast\varphi(u(s))(x)ds = \int_0^{t-\delta} h(s)						(\partial_xp(K(s,t)))\ast\varphi(u(s))(x)ds.		
\end{equation}
Using a classical result on uniform convergence (see Theorem 7.17 in \cite{Rudin}) we have the right hand side of \eqref{igai} converges uniformly on $\mathbb{R}^d$ to the function
	$$x \mapsto \int_0^t (\partial_x\left(p(K(s,t)))\right)\ast\varphi(u(s))(x)ds.$$
The property \eqref{pdd} of $p$ implies that such function is well defined because our hypotheses turns out
	$$\int_0^t h(s)K(s,t)^{-1/\alpha}ds < \infty, \enspace t>0.$$
Inasmuch as
\begin{align*}
	\partial_x\int_0^t h(s)p(K(s,t))\ast\varphi(u(s))(x)ds &= \mathop{\lim}\limits_{\delta\downarrow0}				\partial_x\int_0^{t-\delta} h(s)p(K(s,t))\ast\varphi(u(s))(x)ds\\
	&= \mathop{\lim}\limits_{\delta\downarrow0} \int_0^{t-\delta} h(s)(\partial_xp(K(s,t)))\ast					\varphi(u(s))(x)ds\\
	&= \int_0^t h(s)(\partial_xp(K(s,t)))\ast\varphi(u(s))(x)ds.
\end{align*}

In this way we get, for $0<t\leq T^\star$,
\begin{equation}\label{eipdp}
	\partial_xu(t,x) = (\partial_xp(K(t)))\ast u_0(x) - \int_0^t h(s)(\partial_xp(K(s,t)))\ast					\varphi(u(s))(x)ds,
\end{equation}
with
	$$|||\partial_xu(t)||| \leq c||u_0||_{u}K(t)^{-1/\alpha} + c\varphi(|||u|||)\int_0^t h(s)K(s,t)^{-1/				\alpha}ds.$$

To see that $u(t)\in C_b^2(\mathbb{R}^d)$ we take a $\tilde{t}>0$. Using \eqref{eqint} and \eqref{spp} we have, for $t>\tilde{t}$,
\begin{equation}\label{eqintad}
	u(t,x) = p(K(\tilde{t},t))\ast u(\tilde{t})(x) - \int_{\tilde{t}}^t h(s)p(K(s,t))\ast\varphi(u(s))			(x)ds.
\end{equation}
Proceeding as in the deduction of \eqref{eipdp} we obtain
\begin{equation}\label{epsdp}
	\partial_x^2u(t,x) = \left(\partial_xp(K(\tilde{t},t))\right)\ast\left(\partial_xu(\tilde{t})\right)			(x) - \int_{\tilde{t}}^t h(s)\left(\partial_xp(K(s,t))\right)\ast\varphi'(u(s))\partial_xu(s)				(x)ds.
\end{equation}
The second term on the right hand side is well defined because it is bounded by
\begin{align*}
	c\int_{\tilde{t}}^t &h(s)K(s,t)^{-1/\alpha}\varphi'(|||u|||)\left(K(s)^{-1/\alpha} + \int_0^s 					h(r)K(r,s)^{-1/\alpha}dr\right)dr\\
	&\leq c\left\{K(\tilde{t})\mathop{\sup}\limits_{t\in[0,T^\star]} \int_0^t h(s)K(s,t)^{-1/\alpha}ds					+ \left(\mathop{\sup}\limits_{t\in[0,T^\star]}\int_0^t h(s)K(s,t)^{-1/\alpha}ds
				 \right)^2\right\}.
\end{align*}
Due to $\tilde{t}>0$ is arbitrary we get $u(t)\in C_b^2(\mathbb{R}^d),\; t>0$.
		
Using \eqref{pesd} we can see from \eqref{epsdp} that the differentiability of $\varphi$ can be replaced by the hypothesis
	$$\mathop{\sup}\limits_{t\in[0,T^\star]} \int_0^t h(s)K(s,t)^{-2/\alpha}ds < 0,\enspace \forall t>0.$$

\vspace{0.3cm}				
{\it Step four} (temporal regularity of $u$): We just indicate the main steps, in Section 5.2.2 of \cite{D6V} a detailed proof for a similar case is given. The temporal differentiability of $u$ follows from the temporal differentiability of each term in the right hand side of the equality in \eqref{eqint}. Using \eqref{ddct} the differentiability of the first term is given by
	$$\partial_tp(K(t))\ast u_0(x) = k(t)\Delta_\alpha(p(K(t))\ast u_0)(x).$$

To see the differentiability of the second term we proceed as in the spatial case. Applying the Leibniz's rule, for $\delta>0$, we obtain
\begin{eqnarray}
	\partial_t\int_0^{t-\delta} h(s)p(K(s,t))\ast\varphi(u(s))(x)ds = \int_0^{t-\delta} h(s)\partial_t\left(p(K(s,t))\ast\varphi(u(s))\right)(x)ds	\nonumber \\
		+ h(t-\delta)p(K(t-\delta,t))\ast\varphi(u(t-s))(x).		\label{edrlei}
\end{eqnarray}

Since $\Delta_\alpha$ is a closed operator (it is the infinitesimal generator of a $C_0$ semigroup)
	$$\int_0^{t-\delta} h(s)\partial_t(p(K(s,t))\ast\varphi(u(s)))(x)ds = \Delta_\alpha\left(k(t)					\int_0^{t-\delta} h(s)p(K(s,t))\ast\varphi(u(s))ds\right)(x).$$
	What we have done in the spatial regularity allow us to apply Proposition 1 in \cite{DI}, therefore
	$$\int_0^{t-\delta} h(s)\partial_t(p(K(s,t))\ast\varphi(u(s)))(x)ds$$
converges, as $\delta\to0$, to
	$$\Delta_\alpha\left(k(t)\int_0^t h(s)p(K(s,t))\ast\varphi(u(s))ds\right)(x).$$

The convergence of the second term in the right hand side of \eqref{edrlei} follows from the following estimation
\begin{eqnarray}
&&	|p(K(t-\delta,t))\ast\varphi(u(t-\delta))(x) - \varphi(u(t))(x)| \nonumber \\
&&	\leq	H'(|||u|||)||u(t-\delta)-u(t)||_u 	+ |p(K(t-\delta,t))\ast\varphi(u(t))-\varphi(u(t))|(x).\label{epct}
\end{eqnarray}
In fact, the first term, in above inequality, goes to $0$, when $\delta\to0$, because $u$ is continuous in $t>0$ uniformly in $x$. On the other hand, since $\varphi(u(t))\in C_b(\mathbb{R}^d)$ then the convolution operator properties implies that also the second term in \eqref{epct} goes to $0$, when $\delta\to0$ (in this case $K(t-\delta,t)\to0$).
		
Therefore, using \eqref{eqint} and the linearity of $\Delta_\alpha$ we obtain
\begin{align*}
	\partial_tu(t,x) =& \;k(t)\Delta_\alpha(p(K(t))\ast u_0)(x)\\
	&-k(t)\Delta_\alpha\left(\int_0^t h(s)p(K(s,t))\ast\varphi(u(s))ds\right)(x)-h(t)\varphi(u(t,x))\\
	=& \;k(t)\Delta_\alpha u(t)(x)-h(t)\varphi(u(t,x)).
\end{align*}

\vspace{0.3cm}	
{\it Step five} ($u$ is non-negative): Let us define the function $g:[0,T^\star]\to\mathbb{R}$ as
	$$g(t) = \mathop{\inf}\limits_{x\in\mathbb{R}^d} u(t,x).$$
The function $g$ is well defined because $u(t)\in C_b(\mathbb{R}^d)$, for each $t\in[0,T^{\star}]$. The continuity of $u$ in $t$ uniformly in $x$ implies the continuity of $g$. Hence there exists a $\tilde{t}\in[0,T^\star]$ such that
			$$g(\tilde{t}) = \mathop{\inf}\limits_{t\in[0,T^\star]} g(t).$$

The convexity of $\varphi$ and $\varphi(0)=0$ implies
\begin{equation}\label{cpfc}
	\varphi(t) \leq \frac{\varphi(R)+\varphi(-R)}{R}|z|,\; \forall z\in[-R,R].
\end{equation}

Using \eqref{eqint} and \eqref{cpfc} we obtain
	$$||u(t)||_1 \leq ||u_0||_1 + c\int_0^t h(s)||u(s)||_1ds,\; 0\leq t\leq T^\star.$$
Gronwall's lemma turns out
	$$||u(t)||_1 \leq ||u_0||_1\left(1+c\;H(T^\star)\text{ exp}\left\{cT^\star\mathop{\max}\limits_{t			\in[0,T^\star]}h(t)\right\}\right),\; 0\leq t\leq T^\star.$$
Hence \eqref{cpfc} implies
	$$||\varphi(u(t))||_1 \leq c||u(t)||_1 \leq c ||u_0||_1, \;0\leq t\leq T^\star.$$
From this we deduce
\begin{equation}\label{devan}
	\mathop{\limsup}\limits_{||x||\to\infty}\varphi(u(t,y-x)) = \mathop{\limsup}\limits_{||x||\to\infty}			\varphi(u(t,x)) = 0, \;\forall(t,y)\in[0,T^\star]\times\mathbb{R}^d.
\end{equation}
Moreover, for each $x,y\in\mathbb{R}^d,\; 0\leq s\leq\tilde{t}$, we have
	$$h(s)p(K(s,\tilde{t}),y)\varphi(u(s,y-x)) \leq p(K(s,\tilde{t}),y)\varphi(|||u|||)\mathop{\max}				\limits_{s\in[0,T^\star]}h(s).$$
The integrability of $p(K(s,\tilde{t}))$ allows us to use the dominated convergence theorem, then \eqref{devan} and \eqref{eqint} implies 
\begin{equation}\label{dppde}
	\mathop{\lim}\limits_{||x||\to\infty}|u(\tilde{t},x)|=0.
\end{equation}

Let us suppose that $g(\tilde{t})<0$ and $\tilde{t}>0$. By \eqref{dppde} there exists a $M>0$ such that
\begin{equation}\label{accdu}
	|u(\tilde{t},x)| < \frac{|g(\tilde{t})|}{3},\;\forall||x||>M.
\end{equation}
Since $u(\tilde{t})\in C_b(\mathbb{R}^d)$ there exists a $\tilde{x}\in\mathbb{R}^d$, $||\tilde{x}||\leq M$, such that
	$$u(\tilde{t},\tilde{x}) = \mathop{\inf}\limits_{||x||\leq M}u(\tilde{t},x) =
		\mathop{\inf}\limits_{x\in\mathbb{R}^d}u(\tilde{t},x) = g(\tilde{t}) = 
		\mathop{\inf}\limits_{(t,x)\in[0,T^\star]\times\mathbb{R}^d}u(t,x),$$
where the second equality is due to \eqref{accdu}. This implies $\partial_tu(\tilde{t},\tilde{x})=0$ and $\Delta_\alpha u(\tilde{t},\tilde{x})\leq0$, where we used \eqref{pqpmg}. Then \eqref{edps} implies
$$h(\tilde{t})\varphi(u(\tilde{t},\tilde{x})) = k(\tilde{t})\Delta_\alpha u(\tilde{t},\tilde{x})\leq 0.$$
Therefore $g(\tilde{t})=0$ or $\tilde{t}=0$. Such contradiction implies $u(t,x)\geq0$, for each $(t,x)\in[0,T^\star]\times\mathbb{R}^d$.
	
\vspace{0.3cm}		
{\it Step six} (global existence): The global existence will be proved using the common technique of extend the local solution (see for example the proof of Theorem A in \cite{AB}). By $u_1$ we denote the solution of \eqref{edps} obtained in step one. Let us consider the equation \eqref{edps} with initial condition $u_1(T^\star)$. Since $u_1(T^\star)\geq0$ we deduce from \eqref{eqint} that $||u_1(T^\star)||_u\leq||u_0||_u$. Repeating the step one we obtain a solution $v_2\in	B\left([0,T^{\star\star}]\times\mathbb{R}^d\right)$ of \eqref{eqint} with initial condition $u_1(T^\star)$ and $|||v_2|||\leq R^{\star\star}$, where $R^{\star\star}=||u_1(T^\star)||_u+1$ and
	$$T^{\star\star} < H^{-1}\left(\frac{1}{1+\varphi(R^{\star\star})+(D_l\varphi)(R^{\star\star})}				\right).$$

Inasmuch as $H$ and $\varphi$ are increasing in $(0,\infty)$, and $R^{\star\star}\leq R^\star$, then we	can take $T^{\star\star}=T^\star$. Now combining $u_1$ and $v_2$ we define, for each $x\in\mathbb{R}^d$,
	$$u_2(t,x) = \left\{
		\begin{array}{ll}
			u_1(t,x),& 0\leq t\leq T^\star,\\
		\\ v_2(t-T^\star,x),& T^\star\leq t\leq 2T^\star.
		\end{array}
		\right.$$

The property \eqref{eqintad} of the integral equation \eqref{eqint} implies that $u_2$ is the unique solution of \eqref{eqint} on $[0,2T^\star]$, with initial condition $u_0$. In this way we can apply the steps two-five to obtain the unique non-negative solution of \eqref{edps} on $[0,2T^\star]$. Proceeding inductively we get the global solution to \eqref{edps}. \hfill
\end{proof}

\vspace{0.5cm}
\begin{proof}[Proof of Theorem \ref{pcasi}]
Integrating the equation \eqref{eqint} with respect to $x$ and using \eqref{ppu} we obtain
\begin{equation}\label{itpm}
	\int_{\mathbb{R}^d}u(t,x)dx = \int_{\mathbb{R}^d}u_0(y)dy - \int_0^th(s)\int_{\mathbb{R}^d} 					\varphi(u(s,y))dy\;ds.
\end{equation}

Since $u\geq0$, then $|||u|||\leq||u_0||_u$, therefore \eqref{itpm} implies the function $||u(t)||_1=M(t), t\geq0$, is bounded and monotone decreasing, hence $M(\infty)=\mathop{\lim}_{t\to\infty}M(t)$ exists.

The convexity of $\varphi$ produces
	$$\varphi(x) \leq\frac{\varphi(||u_0||_u)}{||u_0||_u}x,\;\forall x\in[0,||u_o||_u],$$
thus \eqref{itpm} yields
	$$M(t) \geq ||u_0||_1 - \frac{\varphi(||u_0||_u)}{||u_0||_u}\int_0^t h(s)M(s)ds.$$\vspace{-1em}

Let us denote by $y$ the solution of the ordinary differential equation
\begin{equation}\label{edoaux}
	\begin{split}
		\frac{d}{dt}y(t) &= -\frac{\varphi(||u_0||_u)}{||u_0||_u}h(t)y(t),\\
		y(0) &= ||u_0||_1.
	\end{split}
\end{equation}

The comparison theorem for ordinary differential equations implies
	$$M(t) \geq y(t),\;\forall t\geq0.$$
	
Solving \eqref{edoaux} we get \eqref{epmsp} and the last statement of the result is an immediate consequence of \eqref{epmsp}.\hfill
\end{proof}

\vspace{0.5cm}
\begin{proof}[Proof of Theorem \ref{tcpa}]
The proof will be divided in two cases.

\vspace{0.3cm}
{\it Case $p=1$}: Let us take $t>\tilde{t}$, Minkowski inequality yields
\begin{eqnarray}
		||u(t)-M(\infty)p(K(t))||_1 &\leq& ||u(t)-p(K(\tilde{t},t))\ast u(\tilde{t})||_1 \nonumber\\
		&&+ ||p(K(\tilde{t},t)\ast u(\tilde{t}) - M(\tilde{t})p(K(\tilde{t},t))||_1 \nonumber\\
		&&+ ||M(\tilde{t})p(K(\tilde{t},t) - M(\tilde{t})p(K(t))||_1 \nonumber\\
		&&+ ||M(\tilde{t})p(K(t)) - M(\infty)p(K(t))||_1.\label{depu}
	\end{eqnarray}

From \eqref{itpm} and \eqref{eqintad} we have
\begin{equation}\label{epfco}
	\int_{\tilde{t}}^t h(s)||\varphi(u(s))||_1ds \leq ||u(\tilde{t})||_1 \leq ||u_0||_1,\;\forall t>				\tilde{t},
\end{equation}
this and \eqref{eqintad} implies
	$$\mathop{\limsup}\limits_{t\to\infty}||u(t)-p(K(\tilde{t},t))\ast u(\tilde{t})||_1 \leq 						\int_{\tilde{t}}^\infty h(s)||\varphi(u(s))||_1ds.$$

Given that $\int_1^\infty k(s)ds<\infty$ we can apply (g) in Proposition \ref{prop1} to deduce
	$$\mathop{\lim}\limits_{t\to\infty}||p(K(\tilde{t},t))\ast u(\tilde{t}) - M(\tilde{t}) 						p(K(\tilde{t},t))||_1 = 0,\;\forall \tilde{t}>0.$$

Using again (g) of Proposition \ref{prop1} and \eqref{spp} we have
\begin{align*}
	\mathop{\lim}\limits_{t\to\infty}&||M(\tilde{t})p(K(\tilde{t},t)) - M(\tilde{t})p(K(t))||_1\\
	&= M(\tilde{t})\;\mathop{\lim}\limits_{t\to\infty}\left\|p(K(\tilde{t},t))\ast p(K(\tilde{t})) - 				p(K(\tilde{t},t))\int_{\mathbb{R}^d}p(K(\tilde{t}),x)dx\right\|_1 = 0.
\end{align*}
In this way from \eqref{depu} we get
\begin{eqnarray*}
	\mathop{\limsup}\limits_{\tilde{t}\to\infty}\;\mathop{\limsup}\limits_{t\to\infty} ||u(t)-						M(\infty)p(K(t))||_1 &\leq& \mathop{\limsup}\limits_{\tilde{t}\to\infty}\int_{\tilde{t}}^\infty 				h(s)||\varphi(u(s))||_1ds\\
	& & +\mathop{\limsup}\limits_{\tilde{t}\to\infty}|M(\tilde{t})-M(\infty)|=0,
\end{eqnarray*}
where we have used \eqref{epfco} and Theorem \ref{pcasi}.

\vspace{0.3cm}
{\it Case $p>1$}: H\"older's inequality brings about
\begin{equation}\label{xpqn}
	||u(t)-M(\infty)p(K(t))||_p \leq \left\|u(t)-M(\infty)p(K(t))\right\|_1^{1/2p}\left\|u(t)-						M(\infty)p(K(t))\right\|_{2p-1}^{1-1/2p}.
\end{equation}
The Minkowski's inequality and the elementary inequality
	$$(a+b)^\delta \leq 2^\delta(a^\delta+b\delta), \;\forall\delta,a,b\geq0,$$
can be used to obtain
\begin{equation}\label{depp}
	||u(t)-M(\infty)p(K(t))||_{2p-1}^{1-1/2p} \leq 2^{1-1/2p}\left(||u(t)||_{2p-1}^{1-1/2p}+						||M(\infty)p(K(t))||_{2p-1}^{1-1/2p}\right).
\end{equation}
From \eqref{eqint} we see
\begin{equation}\label{ceye}
	0\leq u(t,x)\leq p(K(t))\ast u_0(x),\;\forall(t,x)\in(0,\infty)\times\mathbb{R}^d,
\end{equation}
thus Young's inequality turns out
	$$||u(t)||_{2p-1} \leq ||u_0||_1||p(K(t))||_{2p-1}.$$
Substituting this estimation in \eqref{depp} and using \eqref{dnpv} we get
\begin{equation}\label{eeps}
	||u(t)-M(\infty)p(K(t))||_{2p-1}^{1-1/2p} \leq cK(t)^{-\frac{d}{\alpha}\left(1-\frac{1}{p}\right)},\;			\forall t>0.
\end{equation}
Substituting \eqref{eeps} in \eqref{xpqn} we see that the result is consequence of the case $p=1$. \hfill
\end{proof}

\vspace{0.5cm}
The following result will be fundamental on the study of the asymptotic behavior of the solutions of \eqref{edps}.

\begin{theorem}\label{cpth} (Comparison).
We assume the hypotheses of Theorem \ref{main}. Moreover we assume $u_0$, $v_0\in C_b(\mathbb{R}^d)\cap  L_1(\mathbb{R}^d)$. If $u$ and $v$ are solutions of
	$$\partial_tu = k(t)\Delta_\alpha u(t) - h(t)\varphi(u(t)),\; u(0)=u_0,$$
and
	$$\partial_tv = k(t)\Delta_\alpha v(t) - h(t)\varphi(v(t)),\; v(0)=v_0,$$
respectively, and $v_0\geq u_0 \geq 0$, then $v\geq u\geq0$.
\end{theorem}
\begin{proof}
Because of $u\geq0$ and $v\geq0$, then $|||u|||\leq||u_0||_u$ and $|||v|||\leq||v_0||_u$. Let us take an arbitrary $T>0$. Using the mean value theorem we can find $\eta(t,x)\in(0,||v_0||_u)$ such that
	$$\varphi(v(t,x))-\varphi(u(t,x)) = \varphi'(\eta(t,x))(v(t,x)-u(t,x)).$$
Let us set $w=v-u$, then
\begin{equation}\label{edpw}
	\begin{split}
		\partial_tw(t,x) &= k(t)\Delta_\alpha w(t,x) - c(t,x)w(t,x),\\
		w(0) &= v_0-u_0 \geq 0,
	\end{split}
\end{equation}
where
	$$c(t,x) = h(t)\varphi'(\eta(t,x)).\hspace{5em}$$
The function $c$ is measurable and
	$$|c(t,x)| \leq \mathop{\max}\limits_{s\in[0,T]}h(s)\varphi'(||v_0||_u),\;\forall(t,x)\in[0,T]			\times\mathbb{R}^d.$$

We will consider the function
	$$z(t,x) = w(t,x) e^{\gamma t},\; (t,x)\in[0,T]\times\mathbb{R}^d,$$
where $\gamma=\mathop{\max}\left\{|c(t,x)|:(t,x)\in[0,T]\times\mathbb{R}^d\right\}+1$. From \eqref{edpw} we get
\begin{align*}
	\partial_tz &= k(t)\Delta_\alpha z - (c-\gamma)z,\\
	z(0) &= w(0).
\end{align*}

Proceeding as in step five of Theorem \ref{main} we can find a $(\tilde{t},\tilde{x})\in[0,T]$ such that $z(\tilde{t},\tilde{x})=\mathop{\inf}\left\{z(t,x):(t,x)\in[0,T]\times\mathbb{R}^d\right\}$, then
	$$0 \geq k(\tilde{t})\Delta_\alpha z(\tilde{t},\tilde{x}) = (c(\tilde{t},\tilde{x})-\gamma)z(\tilde{t},			\tilde{x}).$$
From this we deduce $z(\tilde{t},\tilde{x})\geq0$. Hence $w(t,x)\geq0$ for all $(t,x)\in[0,T]\times\mathbb{R}^d$. Then we let $T\to\infty$. \hfill
\end{proof}

\vspace{0.5cm}
As an application of the comparison theorem we give a different condition of Theorem \ref{pcasi} for the convergence of $\lim_{t\to\infty}\int_{\mathbb{R}^d}u(t,x)dx$ when $\varphi$ is a power function.

\begin{proof}[Proof of Theorem \ref{cppo}]
From \eqref{ceye} we have
	$$||u(t)||_\beta^\beta \leq ||p(K(t))\ast u_0||_\beta^\beta.$$
Applying twice the Young's inequality
	$$||p(K(t))\ast u_0||_\beta \leq ||p(K(t))||_\beta||u_0||_1\text{  \ and  \  }								||p(K(t))\ast u_0||_\beta \leq ||p(K(t))||_1||u_0||_\beta.$$
Thus
	$$||u(t)||_\beta^\beta \leq \mathop{\min}\left\{||p(K(t))||_\beta^\beta||u_0||_1^\beta,\enspace 				||p(K(1))||_1^\beta||u_0||_\beta^\beta\right\}.$$
From \eqref{ppu} and \eqref{dnpv} we obtain
\begin{equation}\label{dpbp}
	||u(t)||_\beta^\beta \leq \mathop{\min}\left\{cK(t)^{-\frac{d}{\alpha}\left(\beta-1\right)}						||u_0||_1^\beta, ||u_0||_\beta^\beta\right\}.
\end{equation}

For each $0<\varepsilon\leq1$ we denote by $u_\varepsilon$ the solution of the differential equation \eqref{edps} with initial condition $\varepsilon u_0$. Setting
	$$M_\varepsilon(t) = \int_{\mathbb{R}^d}u_\varepsilon(t,x)dx,\;t\geq0,$$
we see from \eqref{itpm} that
	$$M_\varepsilon(t) = \varepsilon||u_0||_1 - \int_0^t h(s)||u_\varepsilon(s)||_\beta^\beta ds,\;t\geq0.$$
Letting $t\to\infty$,
\begin{equation}\label{ipme}
	M_\varepsilon(\infty) = \varepsilon\left\{||u_0||_1 - \frac{1}{\varepsilon}\int_0^\infty h(s)						||u_\varepsilon(s)||_\beta^\beta ds\right\}.
\end{equation}
Using the inequality \eqref{dpbp} we have
	$$\frac{1}{\varepsilon}\int_0^\infty h(s)||u_\varepsilon(s)||_\beta^\beta ds \leq \varepsilon^{\beta-1}				\int_0^\infty h(s)\mathop{\min}\left\{cK(s)^{-\frac{d}{\alpha}\left(\beta-1\right)}||u_0||_1^\beta,				\;||u_0||_\beta^\beta\right\}ds.$$
Since $\beta>1$, then $-\frac{d}{\alpha}\left(\beta-1\right)<0$, therefore
	$$\frac{1}{\varepsilon}\int_0^\infty h(s)||u_\varepsilon(s)||_\beta^\beta ds \leq c\varepsilon^{\beta-1}			\left\{\int_0^{\tilde{t}}h(s)ds + \int_{\tilde{t}}^\infty h(s)K(s)^{-\frac{d}{\alpha}								\left(\beta-1\right)}ds\right\},$$
for some $\tilde{t}>0$ large enough. This implies
	$$\mathop{\limsup}\limits_{\varepsilon\to0}\frac{1}{\varepsilon}\int_0^\infty h(s)									||u_\varepsilon(s)||_\beta^\beta ds = 0.$$
Due to $||u_0||_1>0$ we can take $0<\varepsilon<||u_0||_1/2$ such that
	$$\frac{1}{\varepsilon}\int_0^\infty h(s)||u_0(s)||_\beta^\beta ds < \frac{||u_0||_1}{2}.$$
Hence, the expression \eqref{ipme} yields
	$$M_\varepsilon(\infty) \geq \frac{\varepsilon}{2}||u_0||_1.$$
The Theorem \ref{cpth} implies $u\geq u_\varepsilon\geq0$, then $M(\infty)\geq M_\varepsilon(\infty)>0$. \hfill
\end{proof}

\subsection*{Acknowledgment}

This work was partially supported by the grant PIM16-4 of Universidad Aut\'onoma de Aguascalientes.

\end{document}